\documentclass[a4paper,10pt]{amsart}
\usepackage{amsmath,amssymb,amscd,amsfonts}
\usepackage{amsthm}
\usepackage{mathdots}
\usepackage{easywncy}
\numberwithin{equation}{section}
\theoremstyle{plain}
\newtheorem{thm}{Theorem}[section]
\newtheorem{prop}[thm]{Proposition}
\newtheorem{lem}[thm]{Lemma}

\newtheorem{cor}[thm]{Corollary}

\newtheorem{exam}[thm]{Example}

\newtheorem{defn}[thm]{Definition}
\newtheorem{rmk}[thm]{Remark}

\newcommand{\F}{\mathbb{F}}

\newcommand{\Q}{\mathbb{Q}}

\newcommand{\R}{\mathbb{R}}
\newcommand{\Z}{\mathbb{Z}}

\newcommand{\cA}{\mathcal{A}}
\newcommand{\cF}{\mathcal{F}}

\newcommand{\frS}{\mathfrak{S}}
\newcommand{\eps}{\epsilon}
\newcommand{\bk}{\Bbbk}
\newcommand{\bx}{\mathbf{x}}
\newcommand{\vep}{\varepsilon}

\newcommand{\Aut}{\mathrm{Aut}}

\newcommand{\even}{\mathrm{even}}

\newcommand{\Gal}{\mathrm{Gal}}

\newcommand{\GL}{\mathrm{GL}}
\newcommand{\Hom}{\mathrm{Hom}}

\newcommand{\odd}{\mathrm{odd}}

\newcommand{\reg}{\mathrm{reg}}

\newcommand{\Sh}{\mathrm{Sh}}
\newcommand{\sh}{\mathrm{sh}}

\newcommand{\Surj}{\mathrm{Surj}}

\newcommand{\vol}{\mathrm{vol}}

\newcommand{\Tri}{\angle}

\newcommand{\sha}{\mathcyr{sh}}

\newcommand{\inj}{\hookrightarrow}

\newcommand{\resp}{resp.\ }

\newcommand{\wh}[1]{\widehat{#1}}

\title[Finite real MZVs generate $Z$]{Finite real multiple zeta values generate the whole space $Z$}
\author{Seidai Yasuda}
\address{Department of Mathematics, Graduate School of Science, 
Osaka university, Tokyonaka, Osaka 560-0043, Japan}
\email{s-yasuda@math.sci.osaka-u.ac.jp}

\begin{document}
\maketitle
\begin{abstract}
We prove that the $\Q$-vector space generated by the multiple
zeta values is generated by the finite real multiple zeta values
introduced by Kaneko and Zagier.
\end{abstract}


\section{Introduction}
Masanobu Kaneko and Don Zagier \cite{KZ} (see also \cite{K})
recently introduced, for each (not necessarily admissible) index 
$\bk = (k_1, \ldots, k_n)$, a real number 
$\zeta^\cF(\bk)$, which they called 
(a regularization with respect to the series
expression of) the finite real multiple zeta value
with index $\bk$.
If we denote by $k = |\bk|$ the weight of $\bk$,
then $\zeta^\cF(\bk)$ belongs to the $\Q$-vector
space $Z_k$ generated by the usual multiple zeta
values of weight $k$.

For each integer $k \ge 0$,
we let $Z^\cF_k$ denote the $\Q$-linear subspace of
$Z_k$ spanned by the finite real multiple
zeta values of weight $k$.
Masanobu Kaneko kindely informed the author that
a numerical experiment by Kaneko and Zagier
suggests that the equality $Z_k^\cF = Z_k$ holds.
The aim of this article is to give a proof of 
this equality,
which is stated in Theorem \ref{thm:main}
as the main result of this article.

Kaneko and Zagier also introduced in \cite{KZ} a variant of
$\zeta^\cF(\bk)$, by using the regularization with respect to 
the iterated integral expression instead of that with respect
to the series expression. We prove in Theorem \ref{thm:main3} 
that an analogous equality holds for this variant.

We prove the equality $Z_k^\cF = Z_k$ in an induction of the pair
$(k,n)$ (with the lexicographic ordering) 
of the weight and the depth of an index.
Let $\bk$ be an index of depth $n$.
It is easy to see that
the finite real multiple zeta value $\zeta^\cF(\bk)$
belongs to the $\Q$-linear subspace $Z_{k,\le n}$ 
of $Z_k$ generated by the multiple zeta values
of weight $k$ and depth at most $n$.
Let $Z_{k,+}$ denote the $\Q$-linear subspace of 
$Z_k$ generated by
$$
\sum_{k', k'' \ge 1 \atop k' + k'' = k}
Z_{k'} \cdot Z_{k''}.
$$
One can show, by using the parity result
(\cite[Corollary 8]{IKZ}, \cite{T}) and \cite[(8.6)]{IKZ}, 
that for any index $\bk$ of depth $n$,
the finite real multiple zeta value $\zeta^\cF(\bk)$
belongs to the $\Q$-linear subspace 
$Z_{k,+} + Z_{k,\le n-1}$.
%
%
We introduce a variant 
$\zeta^{\natural,\cF}(\bk)$ of the finite 
real multiple zeta value $\zeta^\cF(\bk)$.
%
One can check that the $\Q$-linear subspace
of $\R$ generated by the variants 
$\zeta^{\natural, \cF}(\bk)$ 
of weight $k$ coincides with $Z^\cF_k$.
In Proposition \ref{prop:main},
we give an explicit description
of $\zeta^{\natural,\cF}(\bk) \in Z_k$
of weight $k$ and depth $n$
in terms of the multiple zeta values
of weight $n-1$ modulo $Z_{k,+} + Z_{k,\le n-2}$.
%
%
%
%
%
We prove the main theorem by
using this description and some simple linear
algebras.

Among a lot of relations between the multiple zeta values,
we only use the regularized double shuffle relations
in the proof of the main result of this paper.
Hence the main result of this paper can be generalized
without difficulty 
to other multiple zeta values satisfying the
regularized double shuffle relations, e.g.,
to motivic multiple zeta values introduced by Brown \cite{Brown} 
(using an idea of Goncharov) and to $p$-adic multiple zeta values 
introduced by Furusho \cite{Furusho}.

The generalization to the motivic multiple zeta values
will have the following important application in future: 
the author will propose, in his joint work with M.\ Hirose in preparation, 
a conjectural formula between Deligne's $p$-adic multiple zeta values 
and some finite multiple harmonic sums.
If we assume this conjectural formula, then the generalization of the
main result of this paper to the motivic multiple zeta values
gives a surjetive homomorphism of $\Q$-algebras from the $\Q$-linear
span of motivic multiple zeta values modulo motivic $\zeta(2)$
to the $\Q$-subalgebra $Z_\cA$ of 
$\cA = \left( \prod_p \F_p \right) \otimes_\Z \Q$
generated by the multiple zeta values $\zeta^\cA(\bk)$ 
with values in $\cA$ (see \cite{KZ} and \S2.6 for the definition
of $\zeta^\cA(\bk)$).
This will give us a lot of information on the structures of
the $\Q$-algebra $Z_\cA$.

\section{Notation}
%
%

\subsection{Notation for indices}
We denote by $\Z_{\ge 0}$ and $\Z_{\ge 1}$ 
the set of non-negative integers,
and the set of positive integers, respectively.

Let $I$ denote the following set:
$$
I = \coprod_{n \in \Z_{\ge 0}} 
(\overbrace{\Z_{\ge 1} \times 
\cdots \times \Z_{\ge 1}}^{n \text{ times}}).
$$
An element in $I$ is called an index.
For an index $\bk = (k_1,\ldots,k_n)$,
the integer $k_1 + \cdots + k_n$ is called the
weight of $\bk$ and is denoted by $|\bk|$.
(For $n=0$, we understand $|\bk| =0$).

The unique index $\bk$ with $|\bk| =0$ is
called the empty index and is denoted by $\emptyset$.

\subsection{Multiple zeta values}
Let us recall the definition of multiple zeta values
introduced by Hoffman \cite{Hoffman} and independently by
Zagier \cite{Zagier}.
We say that an index $\bk=(k_1,\ldots,k_n)$ 
is an admissible index if
$\bk=\emptyset$ or $k_n \ge 2$.
If $\bk=(k_1,\ldots,k_n)$ is an admissible index,
then the infinite sum
\begin{equation}\label{eq:sum}
\sum_{0 < m_1 < \ldots < m_n}
\frac{1}{m_1^{k_1} \cdots m_n^{k_n}}
\end{equation}
is absolutely convergent 
in the field $\R$ of real numbers.
We denote by $\zeta(\bk)$ or
by $\zeta(k_1,\ldots,k_n)$ the real number
expressed as the infinite sum \eqref{eq:sum}.

\subsection{Finite real multiple zeta values}

Let $\bk = (k_1,\ldots,k_n)$ be an index.
According to Kaneko and Zagier \cite{KZ}, \cite{K}
we define (the regularization with respect to
the series expression of) the finite real 
multiple zeta value
$\zeta^\cF(\bk)=\zeta^\cF(k_1,\ldots,k_n)$ 
as the limit
$$
\zeta^\cF(\bk) = \lim_{M \to \infty}
\sum_{m_1,\ldots,m_n \in \Z
\atop {0 < |m_1|, \ldots, |m_n| < M
\atop \frac{1}{m_1} > \cdots > \frac{1}{m_n}}}
\frac{1}{m_1^{k_1} \cdots m_n^{k_n}}.
$$
We can check that this limit exists in $\R$.
It is clear from the definition that
the finite real multiple zeta values $\zeta^\cF(\bk)$
satisfy the shuffle product formulae with respect
to the series expressions.

\subsection{Finite real multiple zeta values with $\natural$}

Let us introduce the following
open subset $\Tri_{n,\R}$ of $\R^n$:
$$
\Tri_{n,\R} = \left\{
(t_1,\ldots,t_n) \in (\R^\times)^n
\ \left| \ \frac{1}{t_1} > \cdots > \frac{1}{t_n} 
\right. \right\}.
$$
For $\bx = (x_1,\ldots,x_n) \in \R^n$, we let
$w(\bx)$ denote the limit
$$
w(\bx) = \lim_{\vep \to 0}
\frac{\vol(B(\bx,\vep) \cap \Tri_{n,\R})}{
\vol(B(\bx,\vep))}.
$$
Here $B(\bx,\vep)$ denotes the open ball
in $\R^n$ of radius $\vep$ centered at $\bx$.
For sufficiently small $\vep$ (which may depend on $\bx$),
we have $w(\bx) = \frac{\vol(B(\bx,\vep) 
\cap \Tri_{n,\R})}{\vol(B(\bx,\vep))}$, and
$w(\bx)$ is $0$ or a rational number whose inverse is
a positive integer.
If $\bx \in \Tri_{n,\R}$, then we have $w(\bx)=1$.
If $\bx$ does not belong to the closure of
$\Tri_{n,\R}$ in $\R^n$, then we have
$w(\bx) =0$.
For example when $n=2$, we have
$$
w((x_1,x_2)) = \left\{
\begin{array}{ll}
1, & \text{if } 0 < x_1 < x_2, x_1 < x_2 < 0, \text{ or }
x_2 < 0 < x_1, \\
\frac12, & \text{if }x_1 = x_2 \text{ or }
x_1 x_2 = 0, \\
0, \text{otherwise}.
\end{array}
\right.
$$
We define (the regularization with respect to
the series expression of) the finite real
multiple zeta values
$\zeta^{\natural, \cF}(\bk)
=\zeta^{\natural,\cF}(k_1,\ldots,k_n)$ with $\natural$
as the limit
$$
\zeta^{\natural, \cF}(\bk) = \lim_{M \to \infty}
\sum_{m_1,\ldots,m_n \in \Z
\atop 0 < |m_1|, \ldots, |m_n| < M}
\frac{w((m_1,\ldots,m_n))}{m_1^{k_1} \cdots m_n^{k_n}}
$$
in $\R$. We can check that this limit exists in $\R$.

\subsection{A relation between $\zeta^\cF(\bk)$ and $\zeta^{\natural,\cF}(\bk)$}
For integers $m$, $n$ with $1 \le m \le n$,
we denote by $\Surj(n,m)$ the set of
surjective maps from the set $\{1,\ldots,n\}$ 
to the set $\{1,\ldots,m\}$ which preserve
the orderings $\le$.
Let $\frS_n = \Aut(\{1,\ldots,n\})$ be the
$n$-th symmetric group.
For $\phi \in \Surj(n,m)$ we set
$$
G_\phi = \{ \sigma \in \frS_n\ |\ \phi \circ \sigma = \phi
\}.
$$
It follows from the definition that
the group $G_\phi$ is equal to the
direct product
$\prod_{j=1}^m \Aut(\phi^{-1}(j))$.
Since $\Aut(\phi^{-1}(j))$ is isomorphic to
the $\sharp \phi^{-1}(j)$-th symmetric group
for each $j$, we have
$$
\sharp G_\phi = \prod_{j=1}^m (\sharp \phi^{-1}(j))!.
$$

Let $n \ge 1$ be an integer
and let $\bk=(k_1,\ldots,k_n)$ 
be an index of depth $n$.
Let $m$ be an integer with $1 \le m \le n$.
For $\phi \in \Surj(n,m)$, we set
$$
\phi_* \bk = \left( 
\sum_{i \in \phi^{-1}(1)}k_i, 
\sum_{i \in \phi^{-1}(2)}k_i,
\ldots, \sum_{i \in \phi^{-1}(m)}k_i
\right).
$$

\begin{lem} \label{lem1}
We have
$$
\zeta^{\natural,\cF}(\bk) =
\sum_{1 \le m \le n} \sum_{\phi \in \Surj(n,m)} 
\frac{1}{\sharp G_\phi}
\zeta^\cF(\phi_* \bk).
$$
\end{lem}

\begin{exam}
For example when $n=2$, we have
$$
\zeta^{\natural, \cF}(k_1,k_2)
= \zeta^\cF(k_1,k_2) + 
\frac{1}{2} \zeta^\cF(k_1+k_2).
$$
\end{exam}

\subsection{The case of totally odd indices}
We conclude this section by mentioning,
although it will not be used in the main
argument of this article,
that the finite real multiple zeta values
$\zeta^{\natural,\cF}(\bk)$ with $\natural$
have the following remarkable property:

\begin{prop}
Let $\bk =(k_1,\ldots,k_n)$ be a non-empty index.
Suppose that $k_1,\ldots, k_n$ are odd numbers.
Then we have
$\zeta^{\natural,\cF}(\bk) = 0$.
\end{prop}

This is proved by showing the equality
$$
\sum_{m_1,\ldots,m_n \in \Z
\atop 0 < |m_1|, \ldots, |m_n| < M}
\frac{w((m_1,\ldots,m_n))}{m_1^{k_1} \cdots m_n^{k_n}}
= 0
$$
for any real number $M > 0$.

Kaneko and Zagier \cite{KZ} defined,
as the collection of finite multiple zeta values
studied by \cite{Hoffman2} and \cite{Zhao}, 
the multiple zeta value $\zeta^\cA(\bk)$
which belongs to the $\Q$-algebra 
$\cA = \left( \prod_p \F_p \right) \otimes_\Z \Q$.
We can consider the variant
$\zeta^{\natural,\cA}(\bk)$
with $\natural$ for the multiple zeta value $\zeta^\cA(\bk)$.
As is expected from the conjecture by
Kaneko and Zagier on a relation between
the multiple zeta values $\zeta^\cA(\bk)$ in $\cA$
and the finite real multiple zeta values,
we have the following:

\begin{lem}
Let $\bk=(k_1,\ldots,k_n)$ be a non-empty index
such that $k_1, \ldots, k_n$ are odd integers.
Then the multiple zeta value
$\zeta^{\natural, \cA}(\bk)$ with $\natural$
in $\cA$ is equal to zero.
\end{lem}

The lemma is proved by showing that,
for a sufficently large prime number $p$,
the component at $p$ of
$\zeta^{\natural,\cA}(\bk)$ is equal to
$$
\sum_{m_1,\ldots,m_n \in \Z
\atop 0 < |m_1|, \ldots, |m_n| < p/2}
\frac{w((m_1,\ldots,m_n))}{m_1^{k_1} \cdots m_n^{k_n}}
= 0.
$$

\section{A key proposition}

\subsection{Notation}

For an integer $k \ge 0$, we denote by $Z_k$
the $\Q$-vector space generated by the multiple
zeta values of weight $k$.
By $Z_{k,+}$ we denote the
$\Q$-linear subspace
$$
Z_{k,+} = \sum_{k' , k'' \ge 1
\atop k' + k'' = k} Z_{k'} \cdot Z_{k''}
$$
of $Z_k$.  For an integer $n$,
we denote by $Z_{k,\le n}$ the $\Q$-linear
subspace generated by the multiple zeta
values of weight $k$ and depth at most $n$.

Let $\bk = (k_1,\ldots,k_n)$ be an index of
depth $n$.
If $\bk$ is of weight $k$, then
it is known that $\zeta^\cF(\bk)$ and
$\zeta^{\natural,\cF}(\bk)$ belong to
$Z_{k,\le n-1}$.

\subsection{Proposition}

\begin{prop} \label{prop:main}
Let $\bk=(k_1,\ldots,k_n)$ be a
(not necessarily admissible) index of depth $n$.
Let $k=k_1+ \cdots + k_n$ denote the weight of $\bk$.
\begin{enumerate}
\item If $k$ and $n$ have a same parity,
then we have
$\zeta^{\natural,\cF}(\bk) \in Z_{k,+}$.
\item If $k$ and $n$ have different parities and
$k \neq 2$, then
the following congruence holds modulo
$Z_{k,+} + Z_{k,\le n-2}$:
\begin{equation} \label{eq1}
\begin{array}{rl}
\zeta^{\natural,\cF}(\bk) \equiv
& -(-1)^{k_1} 
{\displaystyle \sum_{\ell_2,\ldots,\ell_n \ge 0
\atop \ell_2 + \cdots + \ell_n = k_1}
\prod_{i=2}^n \binom{k_i+\ell_i-1}{\ell_i}
\cdot \zeta(k_2+\ell_2,\ldots,k_n+\ell_n) } \\
& + (-1)^{k_n} 
{\displaystyle 
\sum_{\ell_1,\ldots,\ell_{n-1} \ge 0
\atop \ell_1 + \cdots + \ell_{n-1} = k_n}
\prod_{i=1}^{n-1} \binom{k_i+\ell_i-1}{\ell_i}
\cdot \zeta(k_1+\ell_1,\ldots,k_{n-1}+\ell_{n-1})}.
\end{array}
\end{equation}
Here in the right hand side, multiple zeta values
for non-admissible indices may appear.
We regard such values as the constant terms of
the regularizations with respect to the
series expressions or the iterated integral
expressions.
Since no index of the form $(1,\ldots,1)$
appears in the right hand side,
the constant terms of the two kinds of 
regularizations are congruent
modulo $Z_{k,+}$.
\end{enumerate}
\end{prop}

\begin{rmk}
\begin{enumerate}
\item We have stated Proposition \ref{prop:main}
for finite real multiple zeta values
$\zeta^{\natural,\cF}(\bk)$.
One can define finite motivic multiple zeta values
with $\natural$, and can obtain a result for them
similar to that in Proposition \ref{prop:main} 
for the finite real multiple zeta values
with $\natural$.
\item As we mentioned in Section $1$,
we can easily obtain from Proposition \ref{prop:main}
a description of $\zeta^{\cF}(\bk)$ modulo
$Z_{k,+} + Z_{k,\le n-2}$ (see Corollary \ref{cor:main} below).
\item When $k_{n-1}, k_n \ge 2$, the right hand side of
\eqref{eq1} is equal to the
coefficient of $x^{k_1 + k_n}$ of
$$
\sum_{0 < m_2 < \cdots < m_n}
\frac{x^{k_n}}{
(m_2 + x)^{k_2} \cdots (m_n+x)^{k_n}}
- 
\sum_{0 < m_1 < \cdots < m_{n-1}}
\frac{x^{k_1}}{
(m_1 + x)^{k_1} \cdots (m_{n-1}+x)^{k_{n-1}}}
$$
regarded as a formal power series in $x$.
\item We have described the result
modulo $Z_{k,+} + Z_{k,\le n-2}$.
However the same congruence as in
Proposition \ref{prop:main} holds modulo
\begin{equation} \label{eq2}
Z_{k,\le n-2} + 
\sum_{k',k'', n',n'' \ge 1
\atop k'+k'' = k, n' + n'' = n}
Z_{k',\le n'} \cdot Z_{k'', \le n''}.
\end{equation}
\end{enumerate}
\end{rmk}

Let $\Phi(A,B)$ be the Drinfel'd associator, which is
a non-commutative formal power series with coefficients in $\R$ 
in formal variables $A$, $B$. We slightly change the sign 
convention for $\Phi(A,B)$ in the original paper 
\cite[\S2]{Drinfeld} so that the formal power series $\Phi(A,B)$ relates 
the solutions of the Knizhnik-Zamolodchikov equation
$$
\frac{\partial G(z)}{\partial z} = 
\left(\frac{A}{z} + \frac{B}{1-z}\right)
G(z)
$$
in a neighborhood of $z=0$ and that of $z=1$.
Then for any index $\bk=(k_1,\ldots,k_n)$,
the coefficient of the monomial
$A^{k_n-1}
B A^{k_{n-1}-1} \cdots
B A^{k_1 -1}B$
in $\Phi(A,B)$ is equal to the constant term of
the multiple zeta value of index $\bk$
regularized with respect to
the iterated integral expression.
For a word $w$ over $\{A,B\}$, we denote by $Z(w)$
the coefficient of $w$ in $\Phi(A,B)$.
Under this notation, 
Proposition \ref{prop:main} can be restated as follows:
\begin{cor}\label{cor:main}
Let $\bk=(k_1,\ldots,k_n)$ be a
(not necessarily admissible) index of depth $n$.
Let $k=k_1+ \cdots + k_n$ denote the weight of $\bk$.
\begin{enumerate}
\item If $k$ and $n$ have a same parity,
then we have
$$
\zeta^{\cF}(\bk) 
\equiv - \sum_{i=1}^{n-1}
\zeta(k_1,\ldots,k_{i-1},k_i+k_{i+1},k_{i+1},\ldots, k_n)
$$
modulo $Z_{k,+} + Z_{k,\le n-2}$.
\item If $k$ and $n$ have different parities and
$k \neq 2$, then we have
$$
\zeta^{\cF}(\bk) \equiv
- Z(w) - (-1)^{k} Z(w^*)
$$
modulo $Z_{k,+} + Z_{k,\le n-2}$.
Here
$$
w=A^{k_n-1}
B A^{k_{n-1}-1} \cdots
B A^{k_2-1}
B A^{k_1}
$$
and
$$
w^*=A^{k_1-1}
B A^{k_2 -1} \cdots
B A^{k_{n-1}-1}
B A^{k_n}.
$$
\end{enumerate}
\qed
\end{cor}

\section{Preliminaries for a proof of Proposition \ref{prop:main}}
In this section we give some preliminaries
for the proof of Proposition \ref{prop:main},
which we will give in the next section.

\subsection{Permutation matrices}

For each element $\sigma \in \frS_n$, we denote
by $w_\sigma \in \GL_n(\Z)$ the $n \times n$ matrix
whose $(i,j)$-coordinate is equal to
$\delta_{i,\sigma(j)}$.
For $\sigma, \tau \in \frS_n$, we denote by
$\sigma \tau$ the composite
$\sigma \circ \tau$ of permutations.
We then have 
$w_{\sigma \tau} = w_\sigma w_\tau$.

\subsection{The action of $\GL_n(\Z)$ on formal power series}

Let $R$ be a commutative ring.
For a formal power series 
$f(x_1,\ldots,x_n) \in R[[x_1,\ldots,x_n]]$
in $n$ variables and for $\gamma \in \GL_n(\Z)$,
we denote by $f|_\gamma \in R[[x_1,\ldots, x_n]]$
the formal power series
$$
f|_\gamma (x_1,\ldots,x_n)  = 
f((x_1,\ldots,x_n) \gamma^{-1}).
$$
When $\gamma =w_\sigma$, we denote $f|_{w_{\sigma}}$ by
$f|_\sigma$. By definition, we have
$$
f|_{\sigma}(x_1,\ldots,x_n) =
f(x_{\sigma^{-1}(1)},\ldots,x_{\sigma^{-1}(n)}).
$$

For an element $x = \sum_\gamma a_\gamma \gamma$
in the group ring $\Z[\GL_n(\Z)]$, 
we write $f|_x = \sum_\gamma a_\gamma f|_\gamma$.
For an element
$x = \sum_\sigma a_\sigma \sigma$
in the group ring $\Z[\frS_n]$, 
we write $f|_x = \sum_\sigma a_\sigma f|_\sigma$.
For $x,y \in \Z[\GL_n(\Z)]$ 
(\resp $x,y \in \Z[\frS_n]$), we have
$f|_{xy} = (f|_x)|_y$.

\subsection{A generation function for the multiple zeta values of depth $n$}

For each integer $n \ge 0$, 
let us introduce the following
formal power series
$f^\natural_{\zeta,n}(x_1,\ldots,x_n) \in
\R[[x_1,\ldots,x_n]]$ with
coefficients in $\R$:
$$
f^\natural_{\zeta,n}(x_1,\ldots,x_n)
= \sum_{k_1,\ldots,k_n \ge 1}
\zeta^{\natural,\reg,*}(k_1,\ldots,k_n)
x_1^{k_1-1} \cdots x_n^{k_n-1}.
$$
Here 
$$
\zeta^{\natural,\reg,*}(\bk) =
\sum_{1 \le m \le n} \sum_{\phi \in \Surj(n,m)} 
\frac{1}{\sharp G_\phi}
\zeta^{\reg,*}(\phi_* \bk).
$$
where the notation is as in Lemma \ref{lem1}
and the superscript ``$\reg,*$" of $\zeta^{\reg,*}(\phi_* \bk)$ 
stands for the constant term of the regularization with
respect to the series expression.
In a similar manner, 
we construct the formal power series 
$f^*_{\zeta,n}, f^\sha_{\zeta,n}
\in \R[[x_1,\ldots,x_n]]$ as follows:
$$
f^*_{\zeta,n}(x_1,\ldots,x_n)
= \sum_{k_1,\ldots,k_n \ge 1}
\zeta^{\reg,*}(k_1,\ldots,k_n)
x_1^{k_1-1} \cdots x_n^{k_n-1},
$$
$$
f^\sha_{\zeta,n}(x_1,\ldots,x_n)
= \sum_{k_1,\ldots,k_n \ge 1}
\zeta^{\reg,\sha}(k_1,\ldots,k_n)
x_1^{k_1-1} \cdots x_n^{k_n-1}.
$$
Here the superscript $\reg,\sha$ stands for
the constant term of the regularization with
respect to the iterated integral expression.
When $n=0$, we understand
$f^\natural_{\zeta,0} = f^*_{\zeta,0}
= f^\sha_{\zeta,0} =1$.

\subsection{Notation for shuffles}

For each integer $i$ with $0 \le i \le n$,
let $\Sh_{n,i}$ denote the following 
subset of $\frS_n$:
$$
\Sh_{n,i} = \{ \sigma \in \frS_n\ |\ 
\sigma(1) < \cdots < \sigma(i),
\sigma(i+1) < \cdots < \sigma(n)
\}.
$$
We then have (cf.\ \cite[Theorem 2.5]{Hoffman3})
$$
f^\natural_{\zeta,i}(x_1,\ldots,x_i) 
f^\natural_{\zeta,n-i}(x_{i+1},\ldots,x_n)
=\sum_{\sigma \in \Sh_{n,i}}
f^\natural_{\zeta,n}(x_1,\ldots,x_n)|_\sigma.
$$
We set
$$
\sh_{n,i} = \sum_{\sigma \in \Sh_{n,i}} \sigma.
$$

\subsection{A vector space for the generating functions}

We denote by $Z_k$ the $\Q$-linear subspace of $\R$
generated by the multiple zeta values of weight $k$.
The formal power series $f^\natural_{\zeta,n}$,
$f^*_{\zeta,n}$, $f^\sha_{\zeta,n}$ belong to
the $\Q$-linear subspace
$$
M_n = \left\{ \left.
\sum_{k_1,\ldots,k_n \ge 1} 
a_{k_1,\ldots,k_n } x_1^{k_1-1} \cdots x_n^{k_n-1} \ \right|\ 
a_{k_1,\ldots,k_n} \in Z_{k_1+\cdots + k_n}, 
\forall k_1,\ldots,k_n \ge 1 \right\}
$$
of $\R[[x_1,\ldots, x_n]]$.

\subsection{Some special elements in $\GL_n(\Z)$}

We regard $\frS_n$ as a subgroup of $\frS_{n+1}$ 
in natural way.
Let us introduce the following
matrices $P_n, w_{n,0}, \eps_n \in \GL_n(\Z)$:
$$
P_n = \begin{pmatrix}
1 & 1 & \cdots & 1 \\
0 & 1 & \cdots & 1 \\
\vdots & \ddots & \ddots & \vdots \\
0 & \cdots & 0 & 1
\end{pmatrix},\ \ 
w_{n,0} = \begin{pmatrix}
0 & \cdots & 0 & 1 \\
\vdots & \iddots & \iddots & 0 \\
0 & 1 & \iddots & \vdots \\
1 & 0 & \cdots & 0
\end{pmatrix},\ \ 
\eps_n = - 1_n.
$$

\begin{lem}
Let notation be as above. Then there exists
a unique injective homomorphism
$\iota_n : \frS_{n+1} \inj \GL_n(\Z)$
which satisfies the following properties:
\begin{itemize}
\item If $\sigma \in \frS_n$, 
then $\iota_n(\sigma) = w_\sigma$ holds,
\item The image under $\iota_n$ of the permutation 
$$
\sigma' = \begin{pmatrix} 1 & 2 & \cdots & n-1 & n & n+1 \\
n-1 & n-2 & \cdots & 1 & n+1 & n \end{pmatrix}
$$
is equal to $\eps_n P_n^{-1} w_{n,0} P_n$.
\end{itemize}
\end{lem}

\begin{proof}
Let $\Z^{\oplus n}$ denote the abelian group of
$n$-dimensional column vectors with coordinates in $\Z$.
By the multiplication from the left,
the group $\GL_n(\Z)$ acts from the left on $\Z^{\oplus n}$.
This action give an isomorphism
$\GL_n(\Z) \to \Aut(\Z^{\oplus n})$ of groups.
Via this isomorphism we identify
$\GL_n(\Z)$ with $\Aut(\Z^{\oplus n})$.

Let us consider the embedding
$\Z^{\oplus n} \inj \Z^{\oplus n+1}$
which sends
${}^t (a_1,\ldots,a_n) \in \Z^{\oplus n}$ 
to 
$$
{}^t (a_1,\ldots,a_n,-(a_1+\cdots+a_n)) 
\in \Z^{\oplus n+1}.
$$
The image of this embedding is
stable under the action of the group
$\frS_{n+1}$.
Via this embedding, the group $\frS_{n+1}$ acts
on $\Z^{\oplus n}$.
Let us denote by $\iota_0$ the homomorphism
$\frS_{n+1} \to \Aut(\Z^{\oplus n}) \cong \GL_n(\Z)$
supplied by this action.
It is clear from the definition that we have
$\iota_n(\sigma) = w_\sigma$ for any
$\sigma \in \frS_n$.
Let $e_1,\ldots,e_n \in \Z^{\oplus n}$ be the
standard basis of $\Z^{\oplus n}$.
%
The image $\iota_n(\sigma') \in \Aut(\Z^{\oplus n})$ 
of $\sigma'$ under $\iota_n$ sends 
$e_i$ to $e_{n-i} -e_n$ for
$1 \le i \le n-1$, and sends $e_n$ to $-e_n$.
Hence it is equal to the matrix
$$
\begin{pmatrix}
0 & \cdots & 0 & 1 & 0 \\
\vdots & \iddots & \iddots & \iddots & \vdots \\
0 & \iddots & \iddots & \vdots & \vdots \\
1 & 0 & \cdots & \cdots & 0 \\
-1 & \cdots & \cdots & \cdots & -1
\end{pmatrix}.
$$
One can check that this matrix is equal to
$\eps_n P_n^{-1} w_{n,0} P_n$,
which proves that the homomorphism $\iota_n$
has the desired property.
The uniqueness of $\iota_n$ follows since
the group $\frS_{n+1}$ is generated by $\frS_n$
and the permutation $\sigma'$.
\end{proof}

In the sequel, we often omit the subscript $n$
of $P_n, w_{n,0}, \iota_n$. They are abbreviated by
$P$, $w_0$, $\iota$, respectively.

\subsection{An identity of Ihara-Kaneko-Zagier}

Let $\tau_n \in \frS_{n+1}$ denote the
transposition $\tau_n = (1,n+1)$.
As is mentioned in \cite{IKZ}, the following
equality can be checked easily:
\begin{lem}
The equality
$$
1 + \sh_{n,1} c_{n+1}
= c_{n+1} (1 + \sh_{n,1} \tau_n)
$$
holds in $\Z[\frS_{n+1}]$. Here $c_{n+1}$ denotes the cyclic
permutation
$$
c_{n+1} = \begin{pmatrix}
1 & 2 & \cdots & n-1 & n & n+1 \\
2 & 3 & \cdots & n & n+1 & 1
\end{pmatrix}
$$
of order $n+1$.
\qed
\end{lem}

\section{A proof of Proposition \ref{prop:main}}

In this section, we give a proof of
Proposition \ref{prop:main}.
The claim (1) is easy to prove and is left to the reader.
The rest of this section is devoted to the proof of
the claim (2).

Since
$$
\iota(c_{n+1}) = \eps_n w_0 P^{-1} w_0 P,
$$
we have
$$
f^\natural_{\zeta,n}|_{1+\sh_{n,1} \eps_n w_0 P^{-1} w_0 P}
= f^\natural_{\zeta,n}|_{\eps_n w_0 P^{-1} w_0 P(1+\sh_{n,1} \iota(\tau_n))}.
$$

\subsection{Some subspaces of $M_n$}
For an integer $k \ge 0$,
we denote by $Z_{k,+}$ the subspace
$\sum_{i=1}^{k-1} Z_{i} \cdot Z_{k-i}$ を
of $Z_k$.
The subspace
$$
\left\{\left.\sum_{k_1,\ldots,k_n \ge 1} 
a_{k_1,\ldots,k_n } x_1^{k_1-1} \cdots x_n^{k_n-1} \ \right|\ 
a_{k_1,\ldots,k_n} \in Z_{k_1+\cdots + k_n,+}, 
\forall k_1,\ldots,k_n \ge 1 \right\}
$$
of $M_n$ is denoted by $M_{n,+}$.
We denote by $Z_{k,\le d}$
the $\Q$-linear subspace of $Z_k$ spanned
by the multiple zeta values of weight $k$
and of depth at most $d$
The subspace
$$
\sum_{k_1,\ldots,k_n \ge 1} 
a_{k_1,\ldots,k_n } x_1^{k_1-1} \cdots x_n^{k_n-1} \ |\ 
a_{k_1,\ldots,k_n} \in Z_{k_1+\cdots + k_n, \le d}, 
\forall k_1,\ldots,k_n \ge 1 \}
$$
of $M_n$ is denoted by $M_{n, \le d}$.

\subsection{Step 1}
Since
$f^\natural_{\zeta,n}|_{\sh_{n,1}} \equiv 0 \mod{M_{n,+}}$,
we have
$$
f^\natural_{\zeta,n}|_{1+\sh_{n,1} \eps_n w_0 P^{-1} w_0 P}
\equiv f^\natural_{\zeta,n} \mod{M_{n,+}}.
$$
We can check that $f^\natural_{\zeta,n}|_{1 + (-1)^n w_0} 
\equiv 0 \mod{M_{n,+}}$.
Hence we have
\begin{align*}
& f^\natural_{\zeta,n}|_{\eps_n w_0 P^{-1} w_0 P(1+\sh_{n,1} \iota(\tau_n))} 
\equiv (-1)^{n-1}
f^\natural_{\zeta,n}|_{P^{-1} w_0 P(1+\sh_{n,1} \iota(\tau_n))\eps_n} \\
\equiv & (-1)^{n-1}
(f^\natural_{\zeta,n}
-f^*_{\zeta,n})|_{P^{-1} w_0 P(1+\sh_{n,1} \iota(\tau_n))\eps_n} 
+ (-1)^{n-1}
f^*_{\zeta,n}|_{P^{-1} w_0 P(1+\sh_{n,1} \iota(\tau_n))\eps_n}
\end{align*}
modulo $M_{n,+}$. 
Since $f^*_{\zeta,n} \equiv f^\sha_{\zeta,n}
\mod{M_{n,+} + \Q \zeta(n)}$
(here $\zeta(n)$ is regarded as a constant 
formal power series)
and $f^\sha_{\zeta,n}|_{P^{-1} w_0 P} 
\equiv (-1)^{n-1} f^\sha_{\zeta,n}
\mod{M_{n,+}}$, we have
\begin{align*}
& f^\natural_{\zeta,n}|_{\eps_n w_0 P^{-1} w_0 P(1+\sh_{n,1} \iota(\tau_n))} \\
\equiv & (-1)^{n-1}
(f^\natural_{\zeta,n}
-f^*_{\zeta,n})|_{P^{-1} w_0 P(1+\sh_{n,1} \iota(\tau_n)) \eps_n} + 
f^*_{\zeta,n}|_{(1+\sh_{n,1} \iota(\tau_n)) \eps_n} \\
\equiv & (-1)^{n-1}
(f^\natural_{\zeta,n}-f^*_{\zeta,n})|_{P^{-1} w_0 P(1+\sh_{n,1} \iota(\tau_n))\eps_n} \\
& + 
(f^*_{\zeta,n}-f^\natural_{\zeta,n})|_{(1+\sh_{n,1} \iota(\tau_n))\eps_n} 
+ f^\natural_{\zeta,n}|_{(1+\sh_{n,1} \iota(\tau_n))\eps_n} \\
\equiv & f^\natural_{\zeta,n}|_{\eps_n} 
+ (f^\natural_{\zeta,n}-f^*_{\zeta,n})|_{
((-1)^{n-1} P^{-1} w_0 P-1)(1+\sh_{n,1} \iota(\tau_n))\eps_n}
\end{align*}
modulo $M_{n,+} + \Q \zeta(n)$.
\footnote{We can ignore the index $\bk=(1,\ldots,1)$ in
the discussion here since the weight and the depth have
a same parity for this index.}
Hence we have
\begin{equation}\label{eq:step1}
f^\natural_{\zeta,n}|_{\eps_n}
\equiv
f^\natural_{\zeta,n}
+ (f^\natural_{\zeta,n}-f^*_{\zeta,n})|_{
((-1)^{n-1} P^{-1} w_0 P-1)
(1+\sh_{n,1} \iota(\tau_n))}
\mod{M_{n,+}+\Q \zeta(n)}.
\end{equation}

\subsection{Step 2}
Observe that we have
$$
f^\natural_{\zeta,n}-f^*_{\zeta,n}
\equiv \frac{1}{2} \sum_{i=1}^{n-1}
\frac{1}{x_i-x_{i+1}}
\left(
\begin{array}{l}
f^\natural_{\zeta,n-1}(x_1,\ldots,x_i,
\wh{x_{i+1}},\ldots,x_n) \\
- f^\natural_{\zeta,n-1}(x_1,
\ldots, \wh{x_i},x_{i+1},\ldots,x_n)
\end{array}
\right)
$$
modulo $M_{n,\le n-2}$.
The formal power series
$$
\left. \left(
\frac{f^\natural_{\zeta,n-1}(x_1,\ldots,x_i,
\wh{x_{i+1}},\ldots,x_n)
- f^\natural_{\zeta,n-1}(x_1,
\ldots, \wh{x_i},x_{i+1},\ldots,x_n)}{
x_i -x_{i+1}}
\right)\right|_{P^{-1} w_0 P}
$$
is equal to 
$$
\frac{1}{x_{n-(i+1)} -x_{n-i}}
\left(
\begin{array}{l}
f^\natural_{\zeta,n-1}
(x_n-x_{n-1},\ldots,
\wh{x_n - x_{n-(i+1)}},\ldots,
x_n-x_1,x_n) \\
- f^\natural_{\zeta,n-1}(
x_n-x_{n-1}, 
\ldots, \wh{x_n - x_{n-i}},
\ldots, x_n - x_1,x_n)
\end{array}
\right)
$$
when $i \le n-2$,
and is equal to 
$$
\frac{f^\natural_{\zeta,n-1}
(x_n-x_{n-1},\ldots,x_n -x_{n-i},
\ldots, x_n-x_1)
- f^\natural_{\zeta,n-1}(
x_n-x_{n-1}, 
\ldots, x_n - x_{2},x_n)}{
-x_{1}}
$$
when $i=n-1$.
Since the congruence relations
\begin{align*}
& f^\natural_{\zeta,n-1}(
x_n-x_{n-1}, 
\ldots, \wh{x_n - x_{n-i}},
\ldots, x_n - x_1,x_n) \\
= & f^\natural_{\zeta,n-1}
|_{P_{n-1}^{-1} w_{n-1,0} P_{n-1}}
(x_1, \ldots, \wh{x_{n-i}},
\ldots, x_n) \\
\equiv & f^{\sha}_{\zeta,n-1}
|_{P_{n-1}^{-1} w_{n-1,0} P_{n-1}}
(x_1, \ldots, \wh{x_{n-i}},
\ldots, x_n) \\
\equiv & (-1)^{n-2}
f^\sha_{\zeta,n-1}
(x_1, \ldots, \wh{x_{n-i}},\ldots, x_n) \\
\equiv & (-1)^{n-2}
f^\natural_{\zeta,n-1}
(x_1, \ldots, \wh{x_{n-i}},\ldots, x_n)
\end{align*}
and
\begin{align*}
& f^\natural_{\zeta,n-1}
(x_n-x_{n-1},\ldots,x_n -x_{n-i},
\ldots, x_n-x_1) \\
= & f^\natural_{\zeta,n-1}
|_{P_{n-1}^{-1} w_{n-1,0} P_{n-1}}
(x_2-x_1, \ldots, x_n - x_1) \\
\equiv & f^\sha_{\zeta,n-1}
|_{P_{n-1}^{-1} w_{n-1,0} P_{n-1}}
(x_2-x_1, \ldots, x_n - x_1) \\
\equiv & (-1)^{n-2} f^\sha_{\zeta,n-1}
(x_2-x_1, \ldots, x_n - x_1) \\
\equiv & (-1)^{n-2} f^\natural_{\zeta,n-1}
(x_2-x_1, \ldots, x_n - x_1) \\
\end{align*}
holds modulo $M_{n-1,+} + M_{n-1,\le n-2}$,
the formal power series
$$
\left. \left(
\frac{f^\natural_{\zeta,n-1}(x_1,\ldots,x_i,
\wh{x_{i+1}},\ldots,x_n)
- f^\natural_{\zeta,n-1}(x_1,
\ldots, \wh{x_i},x_{i+1},\ldots,x_n)}{
x_i -x_{i+1}}
\right)\right|_{(-1)^{n-1}P^{-1} w_0 P}
$$
is congruent modulo $M_{n,+} + M_{n,\le n-2}$ to
$$
- \frac{f^\natural_{\zeta,n-1}
(x_1,\ldots,
\wh{x_{n-(i+1)}},\ldots,x_n)
- f^\natural_{\zeta,n-1}
(x_1, \ldots, \wh{x_{n-i}},
\ldots, x_n)}{
x_{n-(i+1)} -x_{n-i}}
$$
when $i \le n-2$, and is congruent
modulo $M_{n,+} + M_{n,\le n-2}$ to
$$
\frac{f^\natural_{\zeta,n-1}
(x_2-x_1,\ldots,x_n -x_1)
- f^\natural_{\zeta,n-1}
(x_2, \ldots, x_n)}{x_{1}}
$$
when $i=n-1$.
Hence we have
\begin{equation}\label{eq:step2}
\begin{array}{rl}
2 (f^\natural_{\zeta,n} - f^*_{\zeta,n})|_{
(-1)^{n-1} P^{-1} w_0 P -1} 
\equiv & \frac{f^\natural_{\zeta,n-1}
(x_2-x_1,\ldots,x_n -x_1)
- f^\natural_{\zeta,n-1}
(x_2, \ldots, x_n)}{
x_{1}} \\
& - \frac{f^\natural_{\zeta,n-1}
(x_1, \ldots,x_{n-1})- 
f^\natural_{\zeta,n-1}
(x_1, \ldots, \wh{x_{n-1}}, x_n)}{
x_{n-1}-x_n}.
\end{array}
\end{equation}
modulo $M_{n,+} + M_{n,\le n-2}$.

\subsection{Step 3}

We apply the operator $\sh_{n,1}$ to the formula \eqref{eq:step2}.

\subsubsection{ }
First we apply $\sh_{n,1}$ to the first term in the right
hand side of \eqref{eq:step2}.
Since
$$
\sh_{n,1} = \sum_{i=1}^n
\begin{pmatrix}
1 & 2 & \cdots & i & i+1 & \cdots & n \\
i & 1 & \cdots & i-1 & i+1 & \cdots & n
\end{pmatrix},
$$
we have
\begin{align*}
& \left. \left( 
\frac{f^\natural_{\zeta,n-1}
(x_2-x_1,\ldots,x_n-x_1)
- f^\natural_{\zeta,n-1}
(x_2,\ldots,x_n)}{x_1}
\right) \right|_{\sh_{n,1}} \\
= & \frac{f^\natural_{\zeta,n-1}
(x_2-x_1,\ldots,x_n - x_1)
- f^\natural_{\zeta,n-1}
(x_2,\ldots,x_n)}{x_1} \\
& + \sum_{i=2}^n
\frac{1}{x_2}
\left(
\begin{array}{l} 
f^\natural_{\zeta,n-1}
(x_3-x_2,\ldots, x_i - x_2, x_1-x_2,x_{i+1}-x_2,
\ldots,x_n-x_2) \\
- f^\natural_{\zeta,n-1}
(x_3,\ldots,x_i,x_1,x_{i+1},\ldots,x_n)
\end{array} \right) \\
= & \frac{f^\natural_{\zeta,n-1}
(x_2-x_1,\ldots,x_n - x_1)
- f^\natural_{\zeta,n-1}
(x_2,\ldots,x_n)}{x_1} \\
& + \frac{1}{x_2}
\left(
(f^\natural_{\zeta,n-1}|_{\sh_{n-1,1}})
(x_1-x_2,x_3-x_2,\ldots,x_n-x_2)
- (f^\natural_{\zeta,n-1}|_{\sh_{n-1,1}})
(x_1,x_3,\ldots,x_n)
\right).
\end{align*}
Hence we have
\begin{equation}\label{eq:step31}
\begin{array}{rl}
& \left. \left( 
\frac{f^\natural_{\zeta,n-1}
(x_2-x_1,\ldots,x_n-x_1)
- f^\natural_{\zeta,n-1}
(x_2,\ldots,x_n)}{x_1}
\right) \right|_{\sh_{n,1}} \\
\equiv & \frac{f^\natural_{\zeta,n-1}
(x_2-x_1,\ldots,x_n - x_1)
- f^\natural_{\zeta,n-1}
(x_2,\ldots,x_n)}{x_1}
\end{array}
\end{equation}
modulo $M_{n,+} + \Q \zeta(n)$.

\subsubsection{ }
Next we apply $\sh_{n,1}$ to the second term
in the right hand side of \eqref{eq:step2}.
We have
\begin{align*}
& \left. \left( 
\frac{f^\natural_{\zeta,n-1}
(x_1,\ldots,x_{n-1})
- f^\natural_{\zeta,n-1}
(x_1,\ldots,\wh{x_{n-1}},x_n)}{
x_{n-1}-x_n}
\right) \right|_{\sh_{n,1}} \\
= & \sum_{i=1}^{n-2}
\frac{f^\natural_{\zeta,n-1}
(x_2,\ldots,x_i,x_1,x_{i+1},\ldots, x_{n-1})
- f^\natural_{\zeta,n-1}
(x_2,\ldots,x_i,x_1,x_{i+1},
\ldots,\wh{x_{n-1}},x_n)}{x_{n-1}-x_n} \\
& + \frac{f^\natural_{\zeta,n-1}
(x_2,\ldots,x_{n-1},x_1)
-f^\natural_{\zeta,n-1}
(x_2,\ldots,x_{n-1},x_n)}{x_1 - x_n} \\
& + \frac{f^\natural_{\zeta,n-1}
(x_2,\ldots,x_n)
- f^\natural_{\zeta,n-1}
(x_2,\ldots,x_{n-1},x_1)}{x_n-x_1} \\
= &
\frac{1}{x_{n-1}-x_n}
\left(
\begin{array}{l}
(f^\natural_{\zeta,n-1}|_{\sh_{n-1,1}}
(x_1,\ldots,x_{n-1})
- f^\natural_{\zeta,n-1}
(x_2,\ldots,x_{n-1},x_1)) \\
- (f^\natural_{\zeta,n-1}|_{\sh_{n-1,1}}
(x_1,\ldots,\wh{x_{n-1}},x_n)
- f^\natural_{\zeta,n-1}
(x_2,\ldots,\wh{x_{n-1}},x_n,x_1))
\end{array}
\right) \\
& + 2 \cdot
\frac{f^\natural_{\zeta,n-1}
(x_2,\ldots,x_{n-1},x_1)
-f^\natural_{\zeta,n-1}
(x_2,\ldots,x_n)}{x_1 - x_n}.
\end{align*}
Hence we have
\begin{equation}\label{eq:step32}
\begin{array}{rl}
& \left. \left( 
\frac{f^\natural_{\zeta,n-1}
(x_1,\ldots,x_{n-1})
- f^\natural_{\zeta,n-1}
(x_1,\ldots,\wh{x_{n-1}},x_n)}{
x_{n-1}-x_n}
\right) \right|_{\sh_{n,1}} \\
\equiv & \frac{1}{x_{n-1}-x_n}
\left(
f^\natural_{\zeta,n-1}
(x_2,\ldots,\wh{x_{n-1}},x_n,x_1)
- f^\natural_{\zeta,n-1}
(x_2,\ldots,x_{n-1},x_1))
\right) \\
& + 2 \cdot
\frac{f^\natural_{\zeta,n-1}
(x_2,\ldots,x_{n-1},x_1)
-f^\natural_{\zeta,n-1}
(x_2,\ldots,x_n)}{x_1 - x_n}
\end{array}
\end{equation}
modulo $M_{n,+} + \Q \zeta(n)$.

\subsubsection{ }
By \eqref{eq:step31} and \eqref{eq:step32},
we have
\begin{equation}\label{eq:step33}
\begin{array}{rl}
& 2(f^\natural_{\zeta,n} - f^*_{\zeta,n})|_{
((-1)^{n-1} P^{-1} w_0 P -1) \sh_{n,1}} \\
\equiv &
\frac{f^\natural_{\zeta,n-1}
(x_2-x_1,\ldots,x_n-x_1)
-f^\natural_{\zeta,n-1}
(x_2,\ldots,x_n)}{x_1} \\
& - \frac{f^\natural_{\zeta,n-1}
(x_2,\ldots,\wh{x_{n-1}},x_n,x_1)
-f^\natural_{\zeta,n-1}
(x_2,\ldots,x_{n-1},x_1)}{x_{n-1}-x_n} \\
& - 2 \cdot
\frac{f^\natural_{\zeta,n-1}
(x_2,\ldots,x_{n-1},x_1)
- f^\natural_{\zeta,n-1}
(x_2,\ldots,x_n)}{x_1-x_n}
\end{array}
\end{equation}
modulo $M_{n,+} + M_{n,\le n-2} + \Q \zeta(n)$.

\subsection{Step 4}
We apply the operator $\iota(\tau_n)$ to the congruence
relation \eqref{eq:step33}.
Since
$$
\iota(\tau_n) =
\begin{pmatrix}
-1 & -1 & \cdots & \cdots & -1 \\
0 & 1 & 0 & \cdots & 0 \\
\vdots & \ddots &  \ddots & \ddots & \vdots \\
\vdots & & \ddots & \ddots & 0 \\
0 & \cdots & \cdots & 0 & 1
\end{pmatrix}
\in \GL_n(\Z)
$$
and $\tau_n^{-1} = \tau_n$, we have
$$
f|_{\tau_n}(x_1,\ldots,x_n)
= f(-x_1,x_2-x_1,\ldots,x_n-x_1)
$$
for a general $f \in \R[[x_1,\ldots,x_n]]$.
Hence we have
\begin{align*}
& 2(f^\natural_{\zeta,n} - f^*_{\zeta,n})|_{
((-1)^{n-1} P^{-1} w_0 P -1) \sh_{n,1}\iota(\tau_n)} \\
\equiv &
\frac{f^\natural_{\zeta,n-1}
(x_2,\ldots,x_n)
-f^\natural_{\zeta,n-1}
(x_2-x_1,\ldots,x_n-x_1)}{-x_1} \\
& - \frac{f^\natural_{\zeta,n-1}
(x_2-x_1,\ldots,\wh{x_{n-1}-x_1},x_n-x_1,-x_1)
-f^\natural_{\zeta,n-1}
(x_2-x_1,\ldots,x_{n-1}-x_1,-x_1)}{
x_{n-1}-x_n} \\
& - 2 \cdot
\frac{f^\natural_{\zeta,n-1}
(x_2-x_1,\ldots,x_{n-1}-x_1,-x_1)
- f^\natural_{\zeta,n-1}
(x_2-x_1,\ldots,x_n-x_1)}{-x_n}
\end{align*}
modulo $M_{n,+} + M_{n,\le n-2} + \Q \zeta(n)$.
We set
$$
Q_{n-1} = \eps_{n-1} w_{n-1,0}
P_{n-1}^{-1} w_{n-1,0} P_{n-1}
= \begin{pmatrix}
-1 & \cdots & \cdots & \cdots & -1 \\
1 & 0 & \cdots & \cdots & 0 \\
0 & \ddots & \ddots & & \vdots \\
\vdots & \ddots & \ddots & \ddots & \vdots \\
0 & \cdots & 0 & 1 & 0
\end{pmatrix} \in \GL_{n-1}(\Z).
$$
We then have
\begin{align*}
& 2(f^\natural_{\zeta,n} - f^*_{\zeta,n})|_{
((-1)^{n-1} P^{-1} w_0 P -1) \sh_{n,1}\iota(\tau_n)} \\
\equiv &
\frac{f^\natural_{\zeta,n-1}
(x_2,\ldots,x_n)
-f^\natural_{\zeta,n-1}
(x_2-x_1,\ldots,x_n-x_1)}{-x_1} \\
& - \frac{f^\natural_{\zeta,n-1}|_{Q_{n-1}^{-1}}
(x_1,\ldots,\wh{x_{n-1}},x_n)
-f^\natural_{\zeta,n-1}|_{Q_{n-1}^{-1}}
(x_1,\ldots,x_{n-1})}{
x_{n-1}-x_n} \\
& - 2 \cdot
\frac{f^\natural_{\zeta,n-1}|_{Q_{n-1}^{-1}}
(x_1,\ldots,x_{n-1})
- f^\natural_{\zeta,n-1}|_{Q_{n-1}^{-1}}
(x_1-x_n,\ldots,x_{n-1}-x_n)}{-x_n}
\end{align*}
modulo $M_{n,+} + M_{n,\le n-2} + \Q \zeta(n)$.
Since
$$
f^\natural_{\zeta,n-1}|_{Q_{n-1}^{-1}}
\equiv f^\natural_{\zeta,n-1}|_{\eps_{n-1}}
$$
modulo $M_{n-1,+} + M_{n-1,\le n-2}$,
we have
\begin{equation}\label{eq:step4}
\begin{array}{rl}
& 2(f^\natural_{\zeta,n} - f^*_{\zeta,n})|_{
((-1)^{n-1} P^{-1} w_0 P -1) \sh_{n,1}\iota(\tau_n)} \\
\equiv &
\frac{f^\natural_{\zeta,n-1}
(x_2,\ldots,x_n)
-f^\natural_{\zeta,n-1}
(x_2-x_1,\ldots,x_n-x_1)}{-x_1} \\
& - \frac{f^\natural_{\zeta,n-1}|_{\eps_{n-1}}
(x_1,\ldots,\wh{x_{n-1}},x_n)
-f^\natural_{\zeta,n-1}|_{\eps_{n-1}}
(x_1,\ldots,x_{n-1})}{
x_{n-1}-x_n} \\
& - 2 \cdot
\frac{f^\natural_{\zeta,n-1}|_{\eps_{n-1}}
(x_1,\ldots,x_{n-1})
- f^\natural_{\zeta,n-1}|_{\eps_{n-1}}
(x_1-x_n,\ldots,x_{n-1}-x_n)}{-x_n}
\end{array}
\end{equation}
modulo $M_{n,+} + M_{n,\le n-2} + \Q \zeta(n)$.

\subsection{Step 5}
By \eqref{eq:step2} and \eqref{eq:step4}, we have
\begin{align*}
& 2(f^\natural_{\zeta,n} - f^*_{\zeta,n})|_{
((-1)^{n-1} P^{-1} w_0 P -1)(1+ \sh_{n,1}\iota(\tau_n))} \\
\equiv &
2 \cdot \frac{f^\natural_{\zeta,n-1}
(x_2-x_1,\ldots,x_n-x_1)
- f^\natural_{\zeta,n-1}
(x_2,\ldots,x_n)}{x_1} \\
& + \frac{f^\natural_{\zeta,n-1}|_{1-\eps_{n-1}}
(x_1,\ldots,\wh{x_{n-1}},x_n)
-f^\natural_{\zeta,n-1}|_{1-\eps_{n-1}}
(x_1,\ldots,x_{n-1})}{x_{n-1}-x_n} \\
& - 2 \cdot
\frac{f^\natural_{\zeta,n-1}|_{\eps_{n-1}}
(x_1-x_n,\ldots,x_{n-1}-x_n)
- f^\natural_{\zeta,n-1}|_{\eps_{n-1}}
(x_1,\ldots,x_{n-1})}{x_n}
\end{align*}
modulo $M_{n,+} + M_{n,\le n-2} + \Q \zeta(n)$.
Hence by \eqref{eq:step1}, we have
\begin{equation}\label{eq:step5}
\begin{array}{rl}
f^\natural_{\zeta,n}|_{\eps_n-1}
\equiv &
\frac{f^\natural_{\zeta,n-1}
(x_2-x_1,\ldots,x_n-x_1)
- f^\natural_{\zeta,n-1}
(x_2,\ldots,x_n)}{x_1} \\
& + \frac{f^\natural_{\zeta,n-1}|_{1-\eps_{n-1}}
(x_1,\ldots,\wh{x_{n-1}},x_n)
-f^\natural_{\zeta,n-1}|_{1-\eps_{n-1}}
(x_1,\ldots,x_{n-1})}{2(x_{n-1}-x_n)} \\
& - 
\frac{f^\natural_{\zeta,n-1}|_{\eps_{n-1}}
(x_1-x_n,\ldots,x_{n-1}-x_n)
- f^\natural_{\zeta,n-1}|_{\eps_{n-1}}
(x_1,\ldots,x_{n-1})}{x_n}
\end{array}
\end{equation}
modulo $M_{n,+} + M_{n,\le n-2} + \Q \zeta(n)$.

\subsection{Step 6}
Let $f^{\natural,\even}_{\zeta,n}$ and
$f^{\natural,\odd}_{\zeta,n}$ denote the even degree part
and the odd degree part of $f^\natural_{\zeta,n}$,
respectively.
By \eqref{eq:step5}, we have
\begin{align*}
f^{\natural,\odd}_{\zeta,n}|_{\eps_n-1}
\equiv &
\frac{f^{\natural,\even}_{\zeta,n-1}
(x_2-x_1,\ldots,x_n-x_1)
- f^{\natural,\even}_{\zeta,n-1}
(x_2,\ldots,x_n)}{x_1} \\
& + \frac{f^{\natural,\even}_{\zeta,n-1}|_{1-\eps_{n-1}}
(x_1,\ldots,\wh{x_{n-1}},x_n)
-f^{\natural,\even}_{\zeta,n-1}|_{1-\eps_{n-1}}
(x_1,\ldots,x_{n-1})}{2(x_{n-1}-x_n)} \\
& - 
\frac{f^{\natural,\even}_{\zeta,n-1}|_{\eps_{n-1}}
(x_1-x_n,\ldots,x_{n-1}-x_n)
- f^{\natural,\even}_{\zeta,n-1}|_{\eps_{n-1}}
(x_1,\ldots,x_{n-1})}{x_n}
\end{align*}
modulo $M_{n,+} + M_{n,\le n-2}$.
Since
$f^{\natural,\odd}_{\zeta,n}|_{\eps_n}
= - f^{\natural,\odd}_{\zeta,n}$ and
$f^{\natural,\even}_{\zeta,n-1}|_{\eps_{n-1}}
= - f^{\natural,\even}_{\zeta,n-1}$, we have
\begin{align*}
- 2 f^{\natural,\odd}_{\zeta,n}
\equiv &
\frac{f^{\natural,\even}_{\zeta,n-1}
(x_2-x_1,\ldots,x_n-x_1)
- f^{\natural,\even}_{\zeta,n-1}
(x_2,\ldots,x_n)}{x_1} \\
& - 
\frac{f^{\natural,\even}_{\zeta,n-1}
(x_1-x_n,\ldots,x_{n-1}-x_n)
- f^{\natural,\even}_{\zeta,n-1}
(x_1,\ldots,x_{n-1})}{x_n}
\end{align*}
modulo $M_{n,+} + M_{n,\le n-2}$.
Let $k_1, \ldots, k_n \ge 1$ be integers.
Set $k := k_1 + \cdots + k_n$.
Suppose that $k \not\equiv n \mod{2}$.
By comparing the coefficients
of $x_1^{k_1-1} \cdots x_n^{k_n-1}$
in the both sides of the equality above,
we have
\begin{align*}
& -2 \zeta^{\natural,\reg,*}(k_1,\ldots,k_n) \\
\equiv & (-1)^{k_1} \sum_{\ell_2, \ldots,\ell_n \ge 0
\atop \ell_2 + \cdots + \ell_n = k_1}
\binom{k_2+\ell_2-1}{\ell_2}
\cdots \binom{k_n+\ell_n-1}{\ell_n}
\zeta^{\natural,\reg,*} 
(k_2+\ell_2,\ldots,k_n+\ell_n) \\
& - (-1)^{k_n} \sum_{\ell_1, \ldots,\ell_{n-1} \ge 0
\atop \ell_1 + \cdots + \ell_{n-1} = k_n}
\binom{k_1+\ell_1-1}{\ell_1}
\cdots \binom{k_{n-1}+\ell_{n-1}-1}{\ell_{n-1}}
\zeta^{\natural,\reg,*}
(k_1+\ell_1,\ldots,k_{n-1}+\ell_{n-1})
\end{align*}
modulo $Z_{k,\le n-2} + \sum_{k',k'' \ge 1
\atop k' + k'' = k} Z_{k'}Z_{k''}$.
Since we have
$\zeta^{\natural,\reg,*}(k_1,\ldots,k_n)
\equiv \zeta^{\reg,*}(k_1,\ldots,k_n)$
modulo $Z_{k,\le n-2}$ and
since we have
$\zeta^{\natural,\cF}(k_1,\ldots,k_n)
\equiv 2 \zeta^{\natural,\reg,*}(k_1,\ldots,k_n)$
modulo $\sum_{k',k'' \ge 1 \atop k' + k'' = k} 
Z_{k'}Z_{k''}$,
we have
\begin{align*}
& \zeta^{\natural,\cF}(k_1,\ldots,k_n) \\
\equiv & (-1)^{k_n} \sum_{\ell_1, \ldots,\ell_{n-1} \ge 0
\atop \ell_1 + \cdots + \ell_{n-1} = k_n}
\binom{k_1+\ell_1-1}{\ell_1}
\cdots \binom{k_{n-1}+\ell_{n-1}-1}{\ell_{n-1}}
\zeta^{\natural,\reg,*}
(k_1+\ell_1,\ldots,k_{n-1}+\ell_{n-1}) \\
& -(-1)^{k_1} \sum_{\ell_2, \ldots,\ell_n \ge 0
\atop \ell_2 + \cdots + \ell_n = k_1}
\binom{k_2+\ell_2-1}{\ell_2}
\cdots \binom{k_n+\ell_n-1}{\ell_n}
\zeta^{\natural,\reg,*} 
(k_2+\ell_2,\ldots,k_n+\ell_n) 
\end{align*}
modulo $Z_{k,\le n-2} + \sum_{k',k'' \ge 1
\atop k' + k'' = k} Z_{k'}Z_{k''}$.
This completes the proof of
Proposition \ref{prop:main} (2).

\section{The main result}

For each integer $k \ge 0$, we denote
by $Z^\cF_k$ the $\Q$-vector space generated by
the finite real multiple zeta values of
weight $k$.
A numerical experiment by Kaneko and Zagier
suggest that $Z^\cF_k = Z_k$.
The main result of this article gives
an affirmative answer to this expectation.

\begin{thm} \label{thm:main}
For any integer $k \ge 0$, we have
$Z^\cF_k = Z_k$.
\end{thm}

\begin{rmk}
By definition we have
$$
\zeta^\cF(k_1,\ldots,k_n)
=\sum_{i=0}^n 
(-1)^{k_i + \cdots + k_n}
\zeta^{\reg,*}(k_1,\ldots,k_i)
\zeta^{\reg,*}(k_n,\ldots,k_{i+1}).
$$
By using
$$
\zeta^{\cF,\sha}(k_1,\ldots,k_n)
= \sum_{i=0}^n 
(-1)^{k_i + \cdots + k_n}
\zeta^{\reg,\sha}(k_1,\ldots,k_i)
\zeta^{\reg,\sha}(k_n,\ldots,k_{i+1})
$$
instead of $\zeta^\cF(k_1,\ldots,k_n)$,
we can construct the $\Q$-vector space
$Z^{\cF,\sha}_k$ analogous to $Z^\cF_k$.
As will be mentioned in \cite{KZ}, we can show that
$\zeta^{\cF,\sha}(k_1,\ldots,k_n) \equiv 
\zeta^{\cF}(k_1,\ldots,k_n)$ modulo
$\zeta(2) Z_{k_1+\cdots + k_n-2}$.
Hence it follows from Theorem \ref{thm:main}
that we have 
$Z^{\cF,\sha}_k + \zeta(2) Z_{k-2} = Z_k$.
\end{rmk}

It is easy to check that Theorem \ref{thm:main}
holds for $k \le 2$.
It follows from Lemma \ref{lem1} that
$Z^\cF_k$ is equal to the $\Q$-vector space
generated by the finite real multiple zeta values
with $\natural$ of weight $k$.
Hence by induction of the pair of
the weight and the depth of an index,
we can reduce Theorem \ref{thm:main} to
the following statement:

\begin{thm} \label{thm:main2}
Let $k ,n \ge 0$ be integers.
When $\bk$ runs over the 
(not necessarily admissible) indices of
weight $k$ and depth $n$,
the right hand sides of \eqref{eq1}
generate the space $Z_{k,\le n-1}$
modulo the space
$$
Z_{k,\le n-2} + 
\sum_{k',k'', n',n'' \ge 1
\atop k'+k'' = k, n' + n'' = n - 1}
Z_{k',\le n'} \cdot Z_{k'', \le n''}.
$$
\end{thm}

\subsection{A variant of Theorem \ref{thm:main}}

Before giving a proof of Theorem \ref{thm:main}, we give a variant of 
Theorem \ref{thm:main} in this paragraph and prove it assuming
Theorem \ref{thm:main}.

For any index $\bk=(k_1,\ldots,k_n)$, 
Kaneko and Zagier introduced in \cite{KZ} 
a variant of $\zeta^\cF(\bk)$, by using the regularization with respect to 
the iterated integral expression instead of that with respect to 
the series expression. According to \cite{KZ}, we denote this variant 
by $\zeta^{\cF,\sha}(\bk)$. By definition we have
$$
\zeta^{\cF,\sha}(\bk) = \sum_{i=0}^n
(-1)^{k_{i+1} + \cdots + k_n}
\zeta^{\reg,\sha}(k_1,\ldots,k_i) \zeta^{\reg,\sha}(k_n,\ldots,k_{i+1}).
$$
For any integers $k \ge 0$, we denote by
$Z_k^{\cF,\sha}$ the $\Q$-vector space generated by
the set $\{\zeta^{\cF,\sha}(\bk)\ |\ |\bk|=k \}$.

\begin{thm} \label{thm:main3}
For any integer $k \ge 0$, we have
$Z^{\cF,\sha}_k = Z_k$.
\end{thm}

We give a proof of Theorem \ref{thm:main3} by assuming
Theorem \ref{thm:main}.

\begin{proof}
By Theorem \ref{thm:main}, it suffices to show that
$\zeta(2) \cdot Z^{\cF,\sha}_k \subset Z^{\cF,\sha}_{k+2}$
holds for any integer $k \ge 0$. The last claim 
follows by applying Proposition \ref{prop:below} below
to the case $\bk'=(2)$.
\end{proof}

\begin{prop} \label{prop:below}
Let $\bk$ and $\bk'$ be indices. Suppose that
all the entries of $\bk'$ are strictly larger than $1$.
Then we have
$$
\zeta^{\cF,\sha}(\bk) \zeta^{\cF}(\bk')
= \sum_{k''} \zeta^{\cF,\sha}(\bk''),
$$
where $\bk''$ runs over 
the stuffles of $\bk$ and $\bk'$.
\end{prop}

\begin{proof}
\newcommand{\QSh}{\mathrm{QSh}}
For two indices $\bk=(k_1,\ldots,k_n)$ and $\bk'=(k'_1,\ldots,k'_{n'})$,
let $\QSh(\bk,\bk')$ denote the multiset of
stuffles of $\bk$ and $\bk'$. For example, we have 
$\QSh((k),(k')) = \{ (k,k'), (k',k), (k+k') \}$ and 
if $k=k'$, then the element $(k,k)$ of $\QSh((k),(k))$
has multiplicity $2$.
For an index $\bk=(k_1,\ldots,k_n)$, let
$I(\bk)$ denote the set 
$$
\{((k_1,\ldots,k_i),(k_n,\ldots,k_{i+1}))\ |\ 0 \le i \le n \}
$$
of pairs of indices.
Then $\zeta^{\cF,\sha}(\bk)$ is equal to the sum
$$
\sum_{(\bk_1,\bk_2) \in I(\bk)} (-1)^{|\bk_2|} \zeta^{\reg,\sha}(\bk_1)
\zeta^{\reg,\sha}(\bk_2).
$$
Now let $\bk=(k_1,\ldots,k_n)$ and
$\bk'=(k'_1,\ldots,k'_{n'})$ be
indices such that all the entries of $\bk'$ are strictly larger than $1$.
Then the product $\zeta^{\cF,\sha}(\bk) \zeta^{\cF}(\bk')$
is equal to the sum
\begin{equation} \label{eq3-0}
\sum_{(\bk_1,\bk_2) \in I(\bk)} 
\sum_{(\bk'_1,\bk'_2) \in I(\bk')}
(-1)^{|\bk_2|+|\bk'_2|}
\zeta^{\reg,\sha}(\bk_1)
\zeta(\bk'_1) \\
\cdot \zeta^{\reg,\sha}(\bk_2) \zeta(\bk'_2).
\end{equation}
It follows from \cite[Theorem 2]{IKZ} that we have
\begin{equation} \label{eq3-1}
\zeta^{\reg,\sha}(\bk_1) \zeta(\bk'_1)
= \sum_{\bk''_1 \in \QSh(\bk_1,\bk'_1)}
\zeta^{\reg,\sha}(\bk''_1)
\end{equation}
and
\begin{equation} \label{eq3-2}
\zeta^{\reg,\sha}(\bk_2)
\zeta(\bk'_2)
= \sum_{\bk''_2 \in \QSh(\bk_2,\bk'_2)}
\zeta^{\reg,\sha}(\bk''_2).
\end{equation}
It is not difficult to check that the equality
\begin{equation} \label{eq3-3}
\coprod_{\bk'' \in \QSh(\bk,\bk')}
I(\bk'')
= \coprod_{(\bk_1,\bk_2) \in I(\bk)}
\coprod_{(\bk'_1,\bk'_2) \in I(\bk')}
\QSh(\bk_1,\bk'_1) \times \QSh(\bk_2,\bk'_2)
\end{equation}
holds for any two indices $\bk$ and $\bk'$, 
where the both sides are regarded as multisets
of pairs of two indices.
We may express,
by using \eqref{eq3-1} and \eqref{eq3-2}, the sum
\eqref{eq3-0} as a sum over the multiset 
which appears in the right hand side of 
\eqref{eq3-3}. By rewriting the sum using \eqref{eq3-3}, 
we obtain the desired
equality.
\end{proof}

\subsection{A dual formulation}

We prove Theorem \ref{thm:main2} by passing to
the dual. Let us fix two integers $k,n \ge 0$
and set 
$$
Z' = Z_{k,\le n-1}
+\sum_{k',k'',n',n'' \ge 1, \atop k'+k'' = k,
n'+n''=n} Z_{k',\le n'} \cdot Z_{k'', \le n''}.
$$
Let $W$ denote the $\Q$-linear subspace of $Z_{k,\le n}$ generated by
the set
$$
\left\{ \text{RHS of (3.1) for }(k_1,\ldots,k_{n+1})
\ \left|\ 
\begin{array}{l}
k_1,\ldots,k_{n+1} \ge 1, \\
k_1+\cdots +k_{n+1}=k 
\end{array}
\right.
\right\}.
$$
Then our aim is to prove $Z_{k, \le n} = Z' + W$.
(The pair $(k,n)$ in this paragraph corresponds to
the pair $(k,n-1)$ in Theorem \ref{thm:main2}.)

We will prove
$Z_{k, \le n}= Z' +W$ by proving that the map 
$\Hom_\Q(Z_{k, \le n},\Q)\to \Hom_\Q(Z' + W,\Q)$ given by 
the restriction of domain is injective.
Let $\alpha: Z_{k, \le n} \to \Q$ be a $\Q$-linear map whose restriction
to $Z' + W$ is equal to zero. We would like to prove $\alpha = 0$.
Let us consider the polynomial
$$
f = \sum_{k_1,\ldots,k_n \ge 1 \atop
k_1 + \cdots + k_n = k}
\alpha(\zeta^{\reg,*}(k_1,\ldots,k_n))
x_1^{k_1-1} \cdots x_n^{k_n-1}
$$
in $x_1, \ldots, x_n$ with coefficients in $\Q$.
It suffices to prove $f=0$.
Set $d=k-n$. By using that $\alpha$ is zero on $Z'$, we can show that
$f$ belongs to $D_{n,d}$.
The condition that $\alpha$ is zero on $W$ implies that
the polynomial $f$ satisfies the condition in the statement of 
Proposition 6.6. Thus we can deduce Theorem 6.3 from
Proposition 6.6.

For each integer $d \ge 0$, we let
$V_{n,d}$ denote the $\Q$-vector space
of homogeneous polynomials in
$x_1, \ldots, x_n$ of degree $d$.
According to \cite{IKZ}, 
we introduced the following $\Q$-linear subspace
$D_{n,d} \subset V_{n,d}$ and call $D_{n,d}$ the
linearized double shuffle space of
weight $n+d$, depth $n$:
$$
D_{n,d} = \{ f \in V_{n,d}\ |\ 
f|_{\sh_{n,i}} = f|_{P_n^{-1} \sh_{n,i}} = 0
\text{ for any $i$ with $1 \le i \le n-1$}
\}.
$$
To prove Theorem \ref{thm:main2},
it suffices to prove the following:

\begin{prop} \label{ques2}
Let $d \ge 0$ be an even integer and
let $f(x_1,\ldots,x_n) \in D_{n,d}$.
Suppose that
the polynomial
\begin{equation} \label{eq2}
\frac{1}{x_1}\left(
f(x_2-x_1,\ldots,x_{n+1} - x_1)-
f(x_2, \ldots, x_{n+1})\right)
\end{equation}
in $n+1$ variable is invariant under the
action of the cyclic permutation of the $n+1$
variables, i.e., the equality
\begin{align*}
& \frac{1}{x_1}\left(
f(x_2-x_1,\ldots,x_{n+1} - x_1)-
f(x_2, \ldots, x_{n+1})\right) \\
= &
\frac{1}{x_{n+1}}\left(
f(x_1-x_{n+1},\ldots,x_n - x_{n+1})-
f(x_1, \ldots, x_n)\right)
\end{align*}
holds. Then $f = 0$.
\end{prop}

\subsection{Preliminaries for the proof of Proposition \ref{ques2}}
For proving Proposition \ref{ques2},
we need some preparation.

\begin{defn}
Let $f(x_1,\ldots,x_n) \in \Q[x_1,\ldots,x_n]$ be
a polynomial in $n$ variables.
We say that $f$ is diagonal translation invariant
if the following equivalent conditions are
satisfied:
\begin{itemize}
\item The polynomial 
$f(x_1 + t, \ldots, x_n +t)$ in $t,x_1,\ldots,x_n$ 
is equal to $f(x_1,\ldots,x_n)$.
\item There exists a polynomial $g(x_1,\ldots,x_{n-1})$ 
in $n-1$ variables satisfying
$$
f(x_1,\ldots,x_n)
= g(x_1 - x_n, \ldots, x_{n-1}-x_n).
$$
\item We have 
$\sum_{i=1}^n \frac{\partial f}{\partial x_i}
= 0$.
\end{itemize}
\end{defn}

\begin{lem} \label{lem_b}
Let $f(x_1,\ldots,x_n)$ be a symmetric polynomial
in $n$ variables. Suppose that
$\frac{\partial}{\partial x_1} (x_1 f)$ is
diagonal translation invariant.
Then $f$ is a constant polynomial.
\end{lem}

\begin{proof}
We may assume that $f$ is homogeneous.
Let $d$ denote the degree of $f$.
Since $f$ is symmetric and
$\frac{\partial}{\partial x_1} (x_1 f)$
is diagonal translation invariant,
$\frac{\partial}{\partial x_i} (x_i f)$ is
diagonal translation invariant for any integer
$i$ with $1 \le i \le n$.
Hence
$$
f = \frac{1}{n+d} \sum_{i=1}^n 
\frac{\partial}{\partial x_i} (x_i f)
$$
is diagonal translation invariant.
Then for any integer $j$ with
$1 \le j \le n$, we have
$$
\frac{\partial f}{\partial x_j}
= \sum_i \frac{\partial}{\partial x_i} 
\frac{\partial}{\partial x_j} (x_j f)
- \frac{\partial}{\partial x_j} 
\left(x_j
\sum_i \frac{\partial f}{\partial x_i} \right)
= 0.
$$
This shows that $f$ is constant.
\end{proof}

\begin{lem} \label{lem_c}
Let $f(x_1,\ldots,x_n)$ be a homogeneous polynomial
of degree $d \ge 2$.
Suppose that
$$
\frac{1}{x_0}(f(x_1-x_0, \ldots, x_n - x_0) - 
f(x_1,\ldots,x_n))
$$
is a symmetric polynomial in $n+1$ variables.
Then $f$ is diagonal translation invariant.
\end{lem}

\begin{proof}
We set $h(x_0,\ldots,x_n) =
\frac{1}{x_0}(f(x_1-x_0, \ldots, x_n - x_0) - 
f(x_1,\ldots,x_n))$.
By assumption $h$ is a symmetric polynomial
of degree $d-1$ and
$$
\frac{\partial}{\partial x_0} (x_0 h)
= \frac{\partial}{\partial x_0}
(f(x_1-x_0, \ldots, x_n - x_0))
$$
is diagonal translation invariant.
Hence it follows from Lemma \ref{lem_b}
and from the assumption $d \ge 2$ 
that we have $h=0$.
Hence $f$ is diagonal translation invariant.
\end{proof}

\begin{lem} \label{lem_d}
Let $d \ge 2$ and $f \in V_{n,d}$.
Suppose that
$\frac{1}{x_{n+1}}
(f(x_1-x_{n+1},\ldots,x_n-x_{n+1})
-f(x_1,\ldots,x_n))$
is invariant under the action of
$\frS_{n+1}$.
Then $f$ is diagonal translation invariant.
\end{lem}

\begin{proof}
We set $h(x_1,\ldots,x_{n+1}) 
= \frac{1}{x_{n+1}}
(f(x_1-x_{n+1},\ldots,x_n-x_{n+1})
-f(x_1,\ldots,x_n))$.
Then the polynomial $h$ is invariant under the
action of the following two elements
$\gamma_0, \gamma_1 \in $
$$
\gamma_0 = \iota_{n+1}(c_{n+1})
= \begin{pmatrix}
0 & \cdots & \cdots & 0 & 1 \\
1 & \ddots &  &  & 0 \\
0 & \ddots & \ddots & & \vdots \\
\vdots & & \ddots & \ddots & \vdots \\
0 & \cdots & 0 & 1 & 0
\end{pmatrix}
$$
and that of
$$
\gamma_1 = \begin{pmatrix}
-1 & \cdots & \cdots & \cdots & -1 \\
0 & 1 & 0 & \cdots & 0 \\
\vdots & \ddots & \ddots & \ddots & \vdots \\
\vdots & & \ddots & \ddots & 0 \\
0 & \cdots & \cdots & 0 & 1
\end{pmatrix}.
$$
It is easy to check that the subgroup 
of $\GL_{n+1}(\Z)$ generated by $\gamma_0$
and $\gamma_1$ is equal to $\iota_{n+1}(\frS_{n+2})$.
In particular $h$ is a symmetric polynomial.
Hence it follows from Lemma \ref{lem_c} 
that $f$ is diagonal translation invariant.
\end{proof}

\begin{lem} \label{lem_a}
Suppose that a polynomial
$f(x_1,\ldots,x_n) \in \Q[x_1,\ldots,x_n]$ 
in $n$ variables satisfies the equality
\begin{align*}
& x_{n+1}
\left(
f(x_2-x_1,\ldots,x_{n+1} - x_1)-
f(x_2, \ldots, x_{n+1})\right) \\
= &
x_1 \left(
f(x_2-x_1,\ldots,x_{n+1} - x_1)
- f(x_1, \ldots, x_n)\right)
\end{align*}
of polynomials in $n+1$ variables.
Then
$\frac{\partial}{\partial x_{n}}
\left(
x_n f(x_1,\ldots,x_n)
\right)$
is diagonal translation invariant.
\end{lem}

\begin{proof}
By assumption
$(x_{n+1}-x_1)f(x_2-x_1,\ldots,x_{n+1} - x_1)$
is the sum of a polynomial in $x_2, \ldots, x_{n+1}$
and a polynomial in
$x_1,\ldots,x_n$.
In particular we have
$$
\frac{\partial}{\partial x_1}
\frac{\partial}{\partial x_{n+1}}
\left(
(x_{n+1}-x_1)f(x_2-x_1,\ldots,x_{n+1} - x_1)
\right) = 0.
$$
This can be rewritten as
$$
\sum_{i=1}^n \frac{\partial}{\partial x_i}
\frac{\partial}{\partial x_{n}}
\left(
x_n f(x_1,\ldots,x_n)
\right) = 0.
$$
Hence
$\frac{\partial}{\partial x_{n}}
\left(
x_n f(x_1,\ldots,x_n)
\right)$
is diagonal translation invariant.
\end{proof}

\subsection{Proof of Proposition \ref{ques2}}

\begin{proof}[Proof of Proposition \ref{ques2}]
Since the claim for $d=0$ is clear.
We may assume $d \ge 2$.
It follows from Lemma \ref{lem_d} that $f$ is
diagonal translation invariant.

Since $f \in D_{n,d}$, we have
$f(x_2-x_1,\ldots,x_{n+1} - x_1)
= f(x_1-x_{n+1},\ldots,x_n - x_{n+1})$.
It then follows from Lemma \ref{lem_a} 
that $\frac{\partial}{\partial x_n} 
(x_n f(x_1,\ldots,x_n))$ is also diagonal
translation invariant.
Hence
$$
x_n \frac{\partial f}{\partial x_n}
= \frac{\partial}{\partial x_n} 
(x_n f(x_1,\ldots,x_n)) - f
$$
is diagonal translation invariant.
Since
$\frac{\partial f}{\partial x_n}$ 
is diagonal translation invariant,
we have $\frac{\partial f}{\partial x_n} = 0$.
%
%
It follows that there exists a polynomial $g$
of $n-1$ variables satisfying
$f(x_1,\ldots,x_n) = g(x_1,\ldots,x_{n-1})$.
On the other hands, the matrix
$$
Q_n = \begin{pmatrix}
-1 & \cdots & \cdots & \cdots & -1 \\
1 & 0 & \cdots & \cdots & 0 \\
0 & \ddots & \ddots & & \vdots \\
\vdots & \ddots & \ddots & \ddots & \vdots \\
0 & \cdots & 0 & 1 & 0
\end{pmatrix} \in \GL_{n}(\Z)
$$
acts on $f$ by $(-1)^d$.
Hence we have
$$
f(x_1,\ldots,x_n)
= (-1)^d f(-x_n,x_1-x_n,\ldots,x_{n-1}-x_n).
$$
Using this equality we can show,
by a backward induction on $i$,
that for each integer $i$ with $0 \le i \le n-1$
there exist a polynomial $g_i$
of $i$ variables satisfying
$f(x_1,\ldots,x_n) = g_i(x_1,\ldots,x_i)$.
Hence $f$ is a constant.
Since we have assumed $d \ge 2$,
we have $f=0$.
\end{proof}

This complete a proof of Theorem \ref{thm:main2},
and hence that of Theorem \ref{thm:main}.

\subsection*{Acknowledgment}
The author is grateful to Masanobu Kaneko,
whose talk at RIMS on July, 2013 have excited
the author's interest in this subject.
He thanks Kentaro Ihara 
for giving me a copy of his notes of the talk
and for a lot of helpful conversations.
He would like to express his sincere gratitude to
Jianqiang Zhao for reading the manuscript carefully
and pointing out a lot of mistakes.
The author was partially supported by JSPS 
Grant-in-Aid for Scientific Research 24540018.

\end{document}